\documentclass{amsart}
\usepackage[dvipdfmx]{graphicx,xcolor}

\usepackage{amsmath}
\usepackage{amssymb}
\usepackage{amsthm}
\usepackage[all]{xy}

\title[Free fermions and Schur expansions of multi-Schur functions]{Free fermions and Schur expansions of multi-Schur functions}
\author{Shinsuke Iwao}
\address{Department of Mathematics, Tokai University, 4-1-1, Kitakaname, Hiratsuka, Kanagawa 259-1292, Japan.}
\email{iwao@tokai.ac.jp}
\date{\today}

\newtheorem{thm}{Theorem}[section]
\newtheorem{prop}[thm]{Proposition}
\newtheorem{lemma}[thm]{Lemma}
\newtheorem{defi}[thm]{Definition}
\newtheorem{example}[thm]{Example}
\newtheorem{rem}[thm]{Remark}

\newtheorem{cor}[thm]{Corollary}

\def\ZZ{\mathord{\mathbb{Z}}}

\newcommand{\bm}[1]{\mbox{\boldmath{$#1$}}}
\def\ket#1{\vert #1 \rangle}
\def\bra#1{\langle #1 \vert}
\def\zet#1{\left\vert{#1}\right\vert}
\def\bx{{\bf x}}
\def\by{{\bf y}}
\def\bp{{\bf p}}
\def\bq{{\bf q}}
\def\bt{{\bf t}}

\begin{document}

\begin{abstract}

Multi-Schur functions are symmetric functions that generalize the supersymmetric Schur functions, the flagged Schur functions, and the refined dual Grothendieck functions, which have been intensively studied by Lascoux. 
In this paper, we give a new free-fermionic presentation of them.
The multi-Schur functions are indexed by a partition and two ``tuples of tuples'' of indeterminates.
We construct a family of linear bases of the fermionic Fock space that are indexed by such data and prove that they correspond to the multi-Schur functions through the boson-fermion correspondence.
By focusing on some special bases, which we call refined bases, we give a straightforward method of expanding a multi-Schur function in the refined dual Grothendieck polynomials.
We also present a sufficient condition for a multi-Schur function to have its Hall-dual function in the completed ring of symmetric functions.


\end{abstract}

\maketitle

\section{Introduction}\label{sec:intro}

In the series of papers \cite{date1982I,date1982II} on classical integrable systems, Date-Jimbo-Miwa showed that the standard basis of the free-fermion Fock space naturally corresponds to the Schur functions through the boson-fermion correspondence.
This result provides a useful tool for studying algebraic properties of symmetric polynomials such as the Jacobi-Trudi formula and the Cauchy identity (See \cite[\S 9]{miwa2012solitons}).
The fermionic description has been generalized by Lam~\cite{lam2006combinatorial} to a wider class of symmetric functions including Hall-Littlewood polynomials and Macdonald polynomials.

In \cite{iwao2020freefermions,iwao2020freefermion,iwao2021neutralfermionic}, the author of this paper gave new fermionic presentations of the (dual) stable Grothendieck polynomials~\cite{lascoux1982structure,lascoux1983symmetry} and the $K$-theoretic $Q$-functions~\cite{ikeda2011bumping}.
A key method used there is to construct linear bases of the Fock space that correspond with desired symmetric functions.

This paper aims to extend the previous results to the multi-Schur functions, which are intensively studied in Lascoux's textbook \cite{lascoux2003symmetric}.
As proved by Motegi-Scrimshaw~\cite{motegi2020refined}, Amanov-Yeliussizov~\cite{amanov2020determinantal}, and Kim~\cite{kim2020jacobitrudi,KIM2021105415}, the multi-Schur functions generalize various important symmetric functions such as the refined dual Grothendieck polynomials~\cite{liu2020refined}.
We show that, as with the stable Grothendieck polynomials, the multi-Schur functions have a useful fermionic description. 
Many algebraic identities such as the Jacobi-Trudi formula are derived from this description.
This result is a fermionic counterpart to the combinatorial/algebraic approaches concerning symmetric Grothendieck polynomials~\cite{Matsumura2017algebraic,matsumura2019flagged,motegi2013vertex,motegi2014k,WheelerZinnJustin+2019+159+195} and their variants~\cite{hwang2021refined,yeliussizov2017duality,YELIUSSIZOV2019453,Yeliussizov2020dual}.

The fermionic presentation is useful in the calculation of the Hall inner product.
Let $\Lambda(X)$ be the ring of symmetric functions in infinitely many variables $X_1,X_2,\dots$ (\S \ref{sec:3.1}).
As noted by various authors (see, for example, \cite[\S 3]{iwao2020freefermion}), we often need a certain extended ring $\widehat{\Lambda}(X)$ of $\Lambda(X)$ to deal with dual functions.
We show that many (but not all) multi-Schur functions have their dual in the extended ring $\widehat{\Lambda}(X)$.
A sufficient condition for a multi-Schur function to have the dual is given (assumption \eqref{eq:assumption_ast} in Lemma \ref{lemma:simple_tuple}).

Each symmetric function in this paper is parameterized by a pair of \textit{tuples of tuples} $\bx/\by$ with
\[
\bx=(x^{(1)},x^{(2)},\dots),\qquad \by=(y^{(1)},y^{(2)},\dots),
\]
where $x^{(i)}=(x^{(i)}_1,x^{(i)}_2,\dots,x^{(i)}_{M_i})$ and $y^{(i)}=(y^{(i)}_1,y^{(i)}_2,\dots,y^{(i)}_{N_i})$ are tuples of finitely many indeterminates (might be empty).
The symbol $\bm{\emptyset}=(\emptyset,\emptyset,\dots)$ will denote the \textit{tuple of empty tuples}.

For $t=(t_1,t_2,\dots)$, we let $t[i]=(t_1,\dots,t_{i-1})$ be the tuple consisting of the first $i-1$ letters of $t$.
In the sequel, we often focus on the tuple of tuples $[\bt]$ that is written as
\begin{equation}\label{eq:refined_tuple}
[\bt]:=(t[1],t[2],\dots).
\end{equation}
We will call a tuple of the form \eqref{eq:refined_tuple} a \textit{refined tuple}.
We prove that every refined tuple possesses the ``orthonormality'' property (Theorem \ref{thm:orthonormalily_of_braket}), which is a key tool for calculating (generalized) Schur expansions.

In this paper, we present two families of symmetric functions: $s_{\lambda}(X)_{\bx/\by}$ and $s_{\lambda}(X)^{\bx}$.
The former is the multi-Schur function in $\Lambda(X)$ and the latter is some new symmetric function in $\widehat{\Lambda}(X)$.
The $s_{\lambda}(X)_{\bx/\by}$ is defined for any $\bx$ and $\by$ while $s_{\lambda}(X)^{\bx}$ is defined only if $\bx$ satisfies a certain finiteness condition (see \eqref{eq:assumption_ast} in Lemma \ref{lemma:simple_tuple}).
We will see that if $\bx=[\bt]$ and $\by=\bm{\emptyset}$ for some refined tuple $[\bt]$, $s_\lambda(X)_{[\bt]}$ and $s_\lambda(X)^{[\bt]}$ are Hall dual with each other (Proposition \ref{prop:dual_functions}).

The following is a list of symmetric functions in our scope.
\begin{center}
\begin{tabular}{c|c|c}
{\bx} and \by & 
$s_{\lambda}(X)_{\bx/\by}$ in $\Lambda(X)$& 
$s_{\lambda}(X)^{\bx}$ in $\widehat{\Lambda}(X)$\\\hline\hline
General $\bx$ and $\by$ &
Multi-Schur function &
---\\\hline
   \begin{tabular}{c}
   $\bx=(x^{(1)},x^{(2)},\dots)$, $\by=\bm{\emptyset}$,\\
   $x^{(i)}=(x_1,\dots,x_{f_i})$,\\
   where $f=(f_1\leq f_2\leq\dots)$\\
   is a flagging
   \end{tabular}&
   \begin{tabular}{c}
   $s_{\lambda}(0)_{\bx}$ is \\
   the flagged \\
   Schur function \\
   in $x_1,x_2,\dots$
   \end{tabular}&
--- \\\hline
   \begin{tabular}{c}
   $\bx$ satisfying \\
   the assumption \eqref{eq:assumption_ast}\\ 
   in Lemma \ref{lemma:simple_tuple},\\
   $\by=\bm{\emptyset}$.
   \end{tabular}&
   \begin{tabular}{c}
   A generalization \\
   of a refined dual \\
   Grothendieck \\
   function
   \end{tabular}&   
   \begin{tabular}{c}
   A generalization \\
   of a refined \\
   Grothendieck \\
   function
   \end{tabular}
\\\hline
   \begin{tabular}{c}
   $\bx=[\bt]$, $\by=\bm{\emptyset}$
   \end{tabular}&
   \begin{tabular}{c}
   Refined dual \\
   Grothendieck \\
   function
   \end{tabular}&   
   \begin{tabular}{c}
   Refined \\
   Grothendieck \\
   function
   \end{tabular}   
\\\hline
   \begin{tabular}{c}
   $\bx=[\bt]$, $\by=\bm{\emptyset}$\\
   with $t=(-\beta,-\beta,\dots)$
   \end{tabular}&
   \begin{tabular}{c}
   Dual stable\\
   Grothendieck \\
   polynomial
   \end{tabular}&   
   \begin{tabular}{c}
   Stable \\
   Grothendieck \\
   polynomial
   \end{tabular}   
\end{tabular}
\end{center}

The paper is organized as follows:
In Section \ref{sec:2}, we give a brief review of free-fermions and the boson-fermion correspondence.
In Section \ref{sec:3}, we introduce the definition of the refined bases and prove their orthonormality, which is the main theorem of this paper.
Section \ref{sec:4} contains free-fermionic presentations of dual functions.
We also prove the Schur expansion of dual functions.
In Section \ref{sec:5}, we extend our result to the skew dual multi-Schur functions.

\section{Free-fermionic presentations of multi-Schur functions}\label{sec:2}

\subsection{Wick's theorem}\label{sec:Wick}

In this section, we review necessary definitions and facts about the boson-fermion correspondence according to the standard textbooks~\cite{kac2013bombay,miwa2012solitons}.
Throughout the paper, we write $[A,B]=AB-BA$ and  $[A,B]_+=AB+BA$.

Assume that $k$ is a field of characteristic $0$.
Let $\mathcal{A}$ be the $k$-algebra of \textit{free fermions} $\psi_n$, $\psi_n^\ast$ ($n\in \ZZ$) with
\begin{equation}\label{eq:free-fermions-relation}
[\psi_m,\psi_n]_+=[\psi^\ast_m,\psi^\ast_n]_+=0,\qquad
[\psi_m,\psi^\ast_n]_+=\delta_{m,n}.
\end{equation}
Let $\ket{0}$, $\bra{0}$ denote the \textit{vacuum vectors}:
\[
\psi_m\ket{0}=\psi^\ast_n\ket{0}=0,\quad
\bra{0}\psi_n=\bra{0}\psi^\ast_m=0,\qquad m< 0,\ n\geq 0.
\]
The \textit{Fock space} (over $k$) is the $k$-space $\mathcal{F}$ generated by the vectors
\begin{equation}\label{eq:elementary-vactors}
\psi_{n_1}\psi_{n_2}\cdots \psi_{n_r}\psi^\ast_{m_1}\psi^\ast_{m_2}\cdots \psi^\ast_{m_s}\ket{0},\
(r,s\geq 0,\ n_1>\dots>n_r\geq 0>m_s>\dots>m_1).
\end{equation}
We also consider the $k$-space $\mathcal{F}^\ast$ generated by the vectors
\begin{equation}\label{eq:elementary-vetcors-dual}
\bra{0}\psi_{m_s}\cdots \psi_{m_2}\psi_{m_1}\psi^\ast_{n_r}\cdots \psi^\ast_{n_2}\psi^\ast_{n_1},\
(r,s\geq 0,\ n_1>\dots>n_r\geq 0>m_s>\dots>m_1).
\end{equation}

The vectors \eqref{eq:elementary-vactors} and \eqref{eq:elementary-vetcors-dual} form bases of $\mathcal{F}$ and $\mathcal{F}^\ast$ respectively (see, for example, \cite[\S 4 and \S 5.2]{kac2013bombay}).
The commutation relation \eqref{eq:free-fermions-relation} determines the left $\mathcal{A}$-module structure of  $\mathcal{F}$ and the right $\mathcal{A}$-module structure of $\mathcal{F}^\ast$.

There exists an anti-algebra involution on $\mathcal{A}$ defined by
\begin{equation}\label{eq:anti-algebra_involution}
{}^\ast:\mathcal{A}\to \mathcal{A};\quad \psi_n\leftrightarrow \psi_n^\ast,
\end{equation}
with $(ab)^\ast=b^\ast a^\ast$ and $(a^\ast)^\ast=a$.
The \textit{transpose} is the $k$-linear involution $\mathcal{F}\leftrightarrow {\mathcal{F}}^\ast$ given by $X\ket{0}\leftrightarrow \bra{0}{X}^\ast$.

Let  
\[
{\mathcal{F}}^\ast\otimes_k\mathcal{F}\to k,\quad 
\bra{w}\otimes \ket{v}\mapsto \langle{w}\vert v\rangle
\]
be the \textit{vacuum expectation value} \cite[\S 4.5]{miwa2012solitons}, that is, a unique $k$-bilinear form with (i) $\langle 0\vert 0\rangle=1$,
(ii) $(\bra{w}\psi_n) \ket{v}=\bra{w} (\psi_n\ket{v})$,
and
(iii) $(\bra{w}\psi_n^\ast) \ket{v}=\bra{w} (\psi_n^\ast\ket{v})$.
From these properties, we may write $\bra{w}X\ket{v}=(\langle{w}\vert X)\ket{v}
=\bra{w}(\vert X\ket{v})
$ for any $X\in \mathcal{A}$.
We often use the abbreviation $\langle X\rangle=\bra{0}X\ket{0}$.


\begin{thm}[Wick's theorem
(see {\cite[\S 2]{alexandrov2013free}, \cite[Exercise 4.2]{miwa2012solitons}})
]\label{thm:Wick}
Let $\{m_1,\dots,m_r\}$ and $\{n_1,\dots,n_{r}\}$ be sets of integers.
Then we have
\[
\langle 
\psi_{m_1}\cdots\psi_{m_{r}}
\psi^\ast_{n_r}\cdots\psi^\ast_{n_{1}}
\rangle
=\det(\langle \psi_{m_i}\psi^\ast_{n_j} \rangle)_{1\leq i,j\leq r}.
\]
\end{thm}

For an integer $m$,  we define the \textit{shifted vacuum vectors} $\ket{m}$, $\bra{m}$ as
\[
\ket{m}=
\begin{cases}
\psi_{m-1}\psi_{m-2}\cdots \psi_0\ket{0}, & m\geq 0,\\
\psi^\ast_{m} \cdots\psi^\ast_{-2}\psi^\ast_{-1}\ket{0}, & m<0
\end{cases}
\]
and $\bra{m}=(\ket{m})^\ast$.
Then, from Wick's theorem, we have:
\begin{cor}\label{cor:dual}
Suppose that $m_1>m_2>\dots>m_{r}\geq -r$ and $n_1>n_2>\dots>n_{s}\geq -r$.
Then we have
\[
\bra{-r}
\psi^\ast_{m_r}\cdots\psi^\ast_{m_{1}}
\psi_{n_1}\cdots\psi_{n_{r}}
\ket{-r}
=
\begin{cases}
1, & \mbox{if } m_i=n_i\mbox{ for all } i,\\
0, & \mbox{otherwise}.
\end{cases}
\]
\end{cor}

%


\subsection{Supersymmetric Schur functions and the boson-fermion correspondence}

We denote by $:\bullet :$ the \textit{normal ordering} (see \cite[\S 2]{alexandrov2013free}, \cite[\S 5.2]{miwa2012solitons}) of free-fermions.
For $m\in \ZZ$, let $a_m=\sum_{i\in \ZZ} :\psi_i\psi^\ast_{i+m}:$ be the \textit{Heisenberg operator} acting on $\mathcal{F}$.
The $a_m$ satisfies the commutation relations~\cite[\S 5.3]{miwa2012solitons}:
\begin{equation}\label{eq:relation_added}
[a_m,a_n]=m\delta_{m+n,0},\qquad
[a_m,\psi_n]=\psi_{n-m},\qquad
[a_m,\psi^\ast_n]=-\psi^\ast_{n+m}.
\end{equation}

Let $x=(x_1,x_2,\dots,x_m)$ and $y=(y_1,y_2,\dots,y_n)$ be tuples of finitely many indeterminates (possibly empty).
We define the \textit{complete supersymmetric functions} $h_i(x/y)$ by the generating function
\[
\sum_{i=0}^\infty h_i(x/y)z^i=\prod_{j=1}^m\frac{1}{1-x_jz}\prod_{l=1}^n(1-y_lz).
\]
When $y_i=0$ for all $i$, the $h_i(x/y)$ reduces to the ordinal complete symmetric polynomial $h_i(x)$.
Since any symmetric polynomial is a polynomial in $h_1(x),h_2(x),\dots$, we define the corresponding supersymmetric function by replacing $h_i(x)$ with $h_i(x/y)$.
For example, the \textit{$i$-th supersymmetric power sum} $p_i(x/y)$ is expressed as $p_i(x/y)=x_1^i+x_2^i+\cdots+x_m^i-y_1^i-y_2^i-\dots-y_n^i$.

Let $s_{\lambda}(x/y)$ be the supersymmetric Schur function~\cite[\S I-5.~Example 23]{macdonald1998symmetric}.
Then they satisfy
\begin{align}
&s_{\lambda}(x/y)=\sum_{\mu}(-1)^{\zet{\lambda-\mu}}s_\mu(x)s_{(\lambda/\mu)'}(y),\qquad \zet{\lambda}=\sum_i\lambda_i,\label{eq:formula1}\\
&s_\lambda(x/y)=(-1)^{\zet{\lambda}}s_{\lambda'}(y/x),\label{eq:formula2}
\end{align}
where $\lambda'$ is the transpose of $\lambda$.
The following equations 
are useful:
\begin{align}
&h_i(x/y)=h_i(x)e_0(y)-h_{i-1}(x)e_1(y)+\cdots+(-1)^ih_0(x)e_i(y),\\
&h_i(\emptyset/y)=(-1)^ie_i(y)=e_i(-y),\qquad\mbox{where}\quad  -y=(-y_1,\dots,-y_n),\\
&s_{\lambda}((x\cup p)/(y\cup q))=\sum_{\mu\subset \lambda}s_{\mu}(x/y)s_{\lambda/\mu}(p/q),\\
&h_n\left((x\cup p)/(y\cup q)\right)=\sum_{i+j=n}h_i(x/y)h_{j}(p/q),
\end{align}
where $x\cup p$ be the tuple consisted of all elements of $x$ and $p$.

Let us consider the formal sums $H(x)$ and $H(x/y)$ defined by
\[
H(x)=\sum_{n>0}\frac{p_n(x)}{n}a_n,\qquad
H(x/y)=\sum_{n>0}\frac{p_n(x/y)}{n}a_n=H(x)-H(y).
\]
If the base field $k$ contains all rational functions in $x_1,\dots,x_m,y_1,\dots,y_n$, the $H(x/y)$ defines a $k$-linear map from $\mathcal{F}$ to itself.
From \eqref{eq:relation_added}, we have the commutation relation
\begin{equation}\label{eq:comm_rel_eH_and_psi_n}
e^{H(x/y)}\psi_ne^{-H(x/y)}=\sum_{i=0}^\infty h_i(x/y)\psi_{n-i}.
\end{equation}

Let $\mathcal{F}^0_{(r)}$ and $\mathcal{F}^0$ be the subspace of $\mathcal{F}$ defined as
\begin{align*}
&\mathcal{F}^0_{(r)}:=\big\{\psi_{n_1}\dots \psi_{n_r}\ket{-r}\in \mathcal{F}\ ;\ 
n_1,n_2,\dots,n_r\in \ZZ
\big\},\\
&\mathcal{F}^0:=\bigcup_{r=0}^\infty\mathcal{F}^0_{(r)}.
\end{align*}
Since $\psi_{-r}\ket{-r}=\ket{-r+1}$, there exists the natural inclusion
\begin{equation}\label{eq:natural_incl}
\mathcal{F}^{0}_{(0)}\subset \mathcal{F}^0_{(1)}\subset \cdots\subset \mathcal{F}^0_{(r)}\subset \cdots \subset \mathcal{F}^0.
\end{equation} 
The \textit{boson-fermion correspondence} is the homomorphism
\begin{equation}\label{eq:bs_corresp}
\ket{v}\mapsto \bra{0}e^{H(x)}\ket{v},\quad x=(x_1,x_2,\dots)
\end{equation}
from $\mathcal{F}^0$ to the ring of symmetric functions $\Lambda(x)$ in $x_1,x_2,\dots$.
This map is in fact isomorphic~\cite[Lemma 9.5]{miwa2012solitons}, \cite[Theorem 6.1]{kac2013bombay}.

For a partition $\lambda=(\lambda_1,\lambda_2,\dots,\lambda_r)$, that is, a weakly decreasing sequence of nonnegative integers, we let
\begin{equation}\label{eq:def_of_ket_lambda}
\ket{\lambda}:=
\ket{\lambda_1,\dots,\lambda_r}=
\psi_{\lambda_1-1}\psi_{\lambda_2-2}\cdots\psi_{\lambda_r-r}\ket{-r}
\in \mathcal{F}^0_{(r)}.
\end{equation}
The image of $\ket{\lambda}$ under \eqref{eq:bs_corresp} coincides with the Schur function $s_\lambda(x)$.
This implies the fact that the composition
\[
\mathcal{F}^0_{(r)}\hookrightarrow \mathcal{F}^0\to \Lambda(x)\stackrel{x_{r+1}=x_{r+2}=\cdots =0}{\to}\Lambda(x_1,\dots,x_r)
\]
is isomorphic, where $\Lambda(x_1,\dots,x_r)$ is the ring of symmetric polynomials in $x_1,\dots,x_r$.

By replacing $h_i(x)$ with $h_i(x/y)$, we obtain the following theorem:
\begin{thm}\label{thm:super_Schur}
We have $s_\lambda(x/y)=\bra{0}e^{H(x/y)}\ket{\lambda}$, where $s_\lambda(x/y)$ is the \textit{supersymmetric Schur function} labeled by $\lambda$.
\end{thm}

We always identify the partition $(\lambda_1,\dots,\lambda_{r},0)$ with $(\lambda_1,\dots,\lambda_{r})$.
This is compatible with the natural inclusion \eqref{eq:natural_incl} since 
$
\ket{\lambda_1,\dots,\lambda_{r-1},0}
=
\ket{\lambda_1,\dots,\lambda_{r-1}}
$.
We let $\ell(\lambda)$ be the length of $\lambda$, that is, the number of non-zero entries of $\lambda$.

\subsection{Multi-Schur functions}\label{sec:etheta_eTheta}

Let $\psi(z)=\sum_{n\in \ZZ}\psi_nz^n$ and $\psi^\ast(z)=\sum_{n\in \ZZ}\psi^\ast_nz^n$ be the generating functions of $\psi_n$ and $\psi^\ast_n$.
They satisfy the commutation relations $[a_n,\psi(z)]=z^n\psi(z)$ and $[a_n,\psi^\ast(z)]=-z^{-n}\psi^\ast(z)$, which are given from \eqref{eq:relation_added}.
By using \eqref{eq:comm_rel_eH_and_psi_n}, we have
\begin{equation}\label{eq:comm_rel_eH_and_psi(z)}
\begin{aligned}
e^{H(x/y)}\psi(z)e^{-H(x/y)}
&= 
\left(
\sum_{i=0}^\infty h_i(x/y)z^i
\right)
\cdot \psi(z)\\
&= \prod_{i=1}^m\frac{1}{1-x_iz}\prod_{j=1}^n(1-y_jz)
\cdot \psi(z)
\end{aligned}
\end{equation}
and
\begin{equation}\label{eq:comm_rel_eH_and_psiast(z)}
\begin{aligned}
e^{H(x/y)}\psi^\ast(z)e^{-H(x/y)}
&= 
\left(
\sum_{i=0}^\infty h_i(y/x)z^{-i}
\right)
\cdot \psi^\ast(z)\\
&= \prod_{i=1}^m(1-x_iz^{-1})\prod_{j=1}^n\frac{1}{1-y_jz^{-1}}
\cdot \psi^\ast(z).
\end{aligned}
\end{equation}

Let $\bx=(x^{(1)},x^{(2)},\dots)$ and $\by=(y^{(1)},y^{(2)},\dots)$ be tuples of tuples where 
\[
x^{(i)}=(x^{(i)}_1,x^{(i)}_2,\dots,x^{(i)}_{M_i}),\quad 
y^{(i)}=(y^{(i)}_1,y^{(i)}_2,\dots,y^{(i)}_{N_i})
\]
are tuples of finitely many indeterminates (see Section \ref{sec:intro}).
For $\ell(\lambda)\leq r$, we define the polynomial 
$S_\lambda(\bx/\by)= S_\lambda(x^{(1)}/y^{(1)},x^{(2)}/y^{(2)},\dots,x^{(r)}/y^{(r)} )$ by
\begin{multline}\label{eq:def_of_mult_Schur}
S_\lambda(\bx/\by):=\\
\bra{0}
\left(e^{H(x^{(1)}/y^{(1)})}\psi_{\lambda_1-1}e^{-H(x^{(1)}/y^{(1)})}\right)
\left(e^{H(x^{(2)}/y^{(2)})}\psi_{\lambda_2-2}e^{-H(x^{(2)}/y^{(2)})}\right)\\
\dots
\left(e^{H(x^{(r)}/y^{(r)})}\psi_{\lambda_r-r}e^{-H(x^{(r)}/y^{(r)})}\right)
\ket{-r}.
\end{multline}

\begin{prop}\label{prop:multi-Schur}
$S_\lambda(\bx/\by)$ is equal to the determinant
\begin{equation}\label{eq:det_formula_multi-Schur}
\det\left(
h_{\lambda_i-i+j}(x^{(i)}/y^{(i)})
\right)_{1\leq i,j\leq r}.
\end{equation}
\end{prop}
\begin{proof}
By Wick's theorem \ref{thm:Wick}, we rewrite the vacuum expectation value \eqref{eq:def_of_mult_Schur} as 
\begin{align*}
&\det\left(
\bra{0}
e^{H(x^{(i)}/y^{(i)})}\psi_{\lambda_i-i}e^{-H(x^{(i)}/y^{(i)})}
\psi_{-j}^\ast\ket{0}
\right)_{1\leq i,j\leq r}\\
&=
\det\left(
\sum_{n=0}^\infty
h_n(x^{(i)}/y^{(i)})
\bra{0}
\psi_{\lambda_i-i-n}
\psi_{-j}^\ast\ket{0}
\right)_{1\leq i,j\leq r}\\
&=\det\left(
h_{\lambda_i-i+j}(x^{(i)}/y^{(i)})
\right)_{1\leq i,j\leq r}.
\end{align*}
\end{proof}

The determinant \eqref{eq:det_formula_multi-Schur} is nothing but the Jacobi-Trudi formula for the \textit{multi-Schur function}~\cite[\S 1.4]{lascoux2003symmetric} (see \cite[\S 2.3]{motegi2020refined} also)\footnote{In the terminology of Lascoux's textbook~\cite{lascoux2003symmetric}, the multi-Schur function $S_{\lambda}(\bx/\by)$ is expressed as 
\[
S_{I}(\mathbb{A}_1-\mathbb{B}_1,\mathbb{A}_2-\mathbb{B}_2,\dots,\mathbb{A}_r-\mathbb{B}_r)
\]
with $I=(\lambda_{r},\lambda_{r-1},\dots,\lambda_{1})$, $\mathbb{A}_i=x^{(r-i+1)}$, and $\mathbb{B}_i=y^{(r-i+1)}$.}.

\begin{example}[Supersymmetric Schur functions]
When
\[
\bx=(x,\emptyset,\emptyset,\dots),\quad
\by=(y,\emptyset,\emptyset,\dots),\quad \mbox{with}\quad
x=(x_1,\dots,x_m),\
y=(y_1,\dots,y_n),
\]
the multi-Schur function $S_\lambda(\bx/\by)$ reduces to the supersymmetric Schur function $s_\lambda(x/y)$.
\end{example}

\begin{example}[Flagged Schur functions]
Let $f=(f_1\leq f_2\leq \dots)$ be a \textit{flagging}, that is, a weakly increasing string of positive integers.
When 
\[
\bx=(x^{(1)},x^{(2)},\dots),\quad
\by=\bm{\emptyset},\quad \mbox{with}\quad
x^{(i)}=(x_1,\dots,x_{f_i}),
\]
the multi-Schur function $S_\lambda(\bx/\by)$ reduces to the \textit{flagged Schur function} $s_\lambda^f(x)$. 
\end{example}

\begin{example}[Refined dual Grothendieck functions]
Motegi-Scrimshaw \cite[\S 3]{motegi2020refined} proved that when 
\[
\bx=(x^{(1)},x^{(2)},\dots),\quad
\by=\bm{\emptyset},\quad \mbox{with}\quad
x^{(i)}=(x_1,\dots,x_n,t_1,\dots,t_{i}),
\]
the multi-Schur function $S_\lambda(\bx/\by)$ reduces to the \textit{refined dual Grothendieck polynomial} $g_\lambda(x;t)$.
It moreover reduces to the ordinal \textit{dual Grothendieck polynomial} when $t_1=\dots=t_r=-\beta$.
\end{example}

Since 
\begin{equation}\label{eq:e^H_is_id}
e^{H(x/y)}\ket{-r}=\ket{-r},
\end{equation}
the expression \eqref{eq:def_of_mult_Schur} is simply rewritten as
\begin{equation}\label{eq:modified_expression}
\bra{0}e^{H(u^{(1)}/v^{(1)} )}\psi_{\lambda_1-1}
e^{H(u^{(2)}/v^{(2)} )}\psi_{\lambda_2-2}\cdots
e^{H(u^{(r)}/v^{(r)} )}\psi_{\lambda_r-r}\ket{-r},
\end{equation}
where $u^{(i)}=x^{(i)}\cup y^{(i-1)}$ and $v^{(i)}=x^{(i-1)}\cup y^{(i)}$.


\begin{rem}
When $x^{(i)}=(x_1,\dots,x_n,\overbrace{-\beta,\dots,-\beta}^{i-1})$ and $\by=\bm{\emptyset}$, \eqref{eq:modified_expression} reduces to
\begin{equation}\label{eq:fermionic-presentation-of-dual-Groth}
\bra{0}e^{H(x_1,\dots,x_n)}\psi_{\lambda_1-1}
e^{H(-\beta)}\psi_{\lambda_2-2}\cdots
e^{H(-\beta)}\psi_{\lambda_r-r}\ket{-r},
\end{equation}
which is exactly the same as the free fermionic presentation of the dual Grothendieck polynomial $g^\beta_\lambda(x_1,\dots,x_n)$ in the author's previous paper \cite[\S 4.1]{iwao2020freefermion}.
\end{rem}

\section{Generalized Schur expansion}\label{sec:3}

In the sequel, we fix the following four tuples of tuples
\[
\bx=(x^{(1)},x^{(2)},\dots),\
\by=(y^{(1)},y^{(2)},\dots),\
\bp=(p^{(1)},p^{(2)},\dots),\
\bq=(q^{(1)},q^{(2)},\dots).
\]
Suppose that all of the indeterminates here are contained in the base field $k$.

\subsection{The symmetric function $s_\lambda(X)_{\bx/\by}$}\label{sec:3.1}

Let $\Lambda(X)$ be the $k$-algebra of symmetric functions in $X_1,X_2,\dots$.
There exists a nondegenerate $k$-bilinear form 
\[
\Lambda(X)\otimes_k\Lambda(X)\to k;\quad f\otimes g\mapsto \langle f,g\rangle \qquad
\mbox{with}\quad
\langle s_\lambda,s_\mu\rangle=\delta_{\lambda,\mu},
\]
that is called the Hall inner product.

\begin{lemma}\label{lemma:fundamental}
Let $\{\ket{\lambda}_1\}_{\lambda}$ and $\{\ket{\lambda}_2\}_{\lambda}$ be two $k$-basis of $\mathcal{F}^0$ labeled by all partitions $\lambda$ (possibly $\lambda=\emptyset$).
Then the following three conditions are equivalent:
\begin{enumerate}
\def\labelenumi{(\roman{enumi})}
\item If $f_\lambda(X)=\bra{0}e^{H(X)}\ket{\lambda}_1$ and $g_\mu(X)=\bra{0}e^{H(X)}\ket{\mu}_2$, then $\langle f_\lambda,g_\mu\rangle=\delta_{\lambda,\mu}$.
\item $\sum_{\lambda}f_\lambda(X)g_\lambda(Y)=\prod_{i,j}(1-X_iY_j)^{-1}$.
\item If $(\ket{\mu}_2)^\ast={}_2\bra{\mu}$, then ${}_2\langle\mu|\lambda\rangle_1=\delta_{\lambda,\mu}$.
\end{enumerate}
\end{lemma}
\begin{proof}
See \cite[I.\S 4]{macdonald1998symmetric} for proofs.
\end{proof}

For a partition $\lambda$ and an integer $r\geq \ell(\lambda)$, let
\begin{multline}\label{eq:def_of_ket_dx}
\ket{\lambda}_{\bx/\by}:=
\left(e^{H(x^{(1)}/y^{(1)} )}\psi_{\lambda_1-1}e^{-H(x^{(1)}/y^{(1)} )}\right)
\left(e^{H(x^{(2)}/y^{(2)} )}\psi_{\lambda_2-2}e^{-H(x^{(2)}/y^{(2)} )}\right)\\
\cdots\left(e^{H(x^{(r)}/y^{(r)} )}\psi_{\lambda_r-r}e^{-H(x^{(r)}/y^{(r)} )}\right)
\ket{-r}.
\end{multline}
The definition \eqref{eq:def_of_ket_dx} does not depend on the choice of $r$ because  \[
e^{H(x/y)}\psi_{-s}e^{-H(x/y)}\ket{-s}=\psi_{-s}\ket{-s}=\ket{-s-1}.
\]
Since the vector $\ket{\lambda}_{\bx/\by}$ is expanded as
\[
\ket{\lambda}_{\bx/\by}=\ket{\lambda}+\sum_{\zet{\mu}<\zet{\lambda}}\alpha_\mu\ket{\mu},\qquad \alpha_\mu\in k,
\]
the set $\{\ket{\lambda}_{\bx/\by}\}_{\lambda}$ forms a $k$-basis of $\mathcal{F}^0$.

Define the new symmetric function $s_\lambda(X)_{\bx/\by}\in \Lambda(X)$ by
\begin{equation}\label{eq:def_of_sxy}
s_\lambda(X)_{\bx/\by}:=\bra{0}e^{H(X)}\ket{\lambda}_{\bx/\by}.
\end{equation}
When $X_i=0$ for all $i$, $s_\lambda(X)_{\bx/\by}$ reduces to the multi-Schur function $S_\lambda(\bx/\by)$.

\subsection{Refined bases and orthonormality}

What we are interested in is the transition matrix between the two bases $\{\ket{\lambda}_{\bx/\by}\}$ and $\{\ket{\mu}_{\bp/\bq}\}$ of $\mathcal{F}^0$.
Unfortunately, we have no explicit formula for general bases at this stage. Instead, we present a certain subclass of bases whose transition matrices are directly calculated from our fermionic presentation.

Set $t=(t_1,t_2,\dots)$.
We consider the refined tuple $[\bt]=(t[1],t[2],\dots)$, which we have defined in  \eqref{eq:refined_tuple}.
When $\bx=[\bt]$ and $\by=\bm{\emptyset}$, the vector $\ket{\lambda}_{\bx/\by}$ \eqref{eq:def_of_ket_dx} reduces to
\begin{equation}\label{eq:1_refined_form}
\begin{split}
\ket{\lambda}_{[\bt]}=
\psi_{\lambda_1-1}
e^{H(t_1)}\psi_{\lambda_2-2}
e^{H(t_2)} 
\cdots
e^{H(t_{r-1})}\psi_{\lambda_r-r}
e^{H(t_{r})}
\ket{-r}.
\end{split}
\end{equation}
To discuss the ``orthonormality'' property, we also define the transposed vector
\begin{equation}\label{eq:def_of_bra_bt}
{}_{[\bt]}\bra{\lambda}:=\bra{-r}
e^{-H(t_r)}\psi^\ast_{\lambda_r-r}
e^{-H(t_{r-1})}\cdots 
e^{-H(t_2)}\psi^\ast_{\lambda_2-2}
e^{-H(t_1)}\psi^\ast_{\lambda_1-1}
\end{equation}
for sufficiently large $r$.
The definition \eqref{eq:def_of_bra_bt} may seem to depend on the choice of $r$. 
However, we will prove later that it is in fact independent of the choice of $r$ when $r$ is sufficiently large. 
See Remark \ref{rem:stability_of_bra}.

The following is the main theorem of this paper:
\begin{thm}[Orthonormality of refined basis]
\label{thm:orthonormalily_of_braket}
We have 
\[
{}_{[\bt]}\langle \mu|\lambda\rangle_{[\bt]}=
\delta_{\lambda,\mu}.
\]
\end{thm}
\begin{proof}
The proof is essentially the same as the proof of \cite[Proposition 4.1]{iwao2020freefermion}.
We here note the equations
\begin{align}
&e^{-H(t_{i})}\psi_{\lambda_i-i}e^{H(t_{i})}=\psi_{\lambda_i-i}
-t_i\psi_{\lambda_i-i-1},\\
&e^{-H(t_{i})}\psi^\ast_{\lambda_i-i}e^{H(t_{i})}=\psi^\ast_{\lambda_i-i}
+t_i\psi^\ast_{\lambda_i-i+1}+t_i^2\psi^\ast_{\lambda_i-i+2}+\cdots,
\end{align}
which are obtained from (\ref{eq:comm_rel_eH_and_psi(z)}--\ref{eq:comm_rel_eH_and_psiast(z)}).

Let $0\leq s\leq r$.
We temporally use the notation
\[
\begin{split}
\ket{\lambda,s}_{[\bt]}=
\psi_{\lambda_{1}-1}
e^{H(t_{1})}\psi_{\lambda_{2}-2}
e^{H(t_{2})} 
\cdots
\psi_{\lambda_s-s}
e^{H(t_{s})}
\ket{-r}
\end{split}
\]
and prove the theorem by induction on $s$.
The theorem follows from the three facts:
\begin{enumerate}
\item[ (I).] If $N\geq \lambda_1$, then $\psi^{\ast}_N\ket{\lambda,s}_{[\bt]}=0$.
\item[ (II).] If $N\geq \mu_1$, then ${}_{[\bt]}\bra{\mu,s}\psi_N=0$.
\item[ (III).] If $\lambda_1=\mu_1$, $\lambda'=(\lambda_2-1,\dots,\lambda_r-1)$, $\mu'=(\mu_2-1,\dots,\mu_r-1)$, and $\bt'=(t_2,t_3,\dots)$.
Then \[
{}_{[\bt]}\langle\mu,s|\lambda,s\rangle_{[\bt]}={}_{[\bt']}\langle\mu',s-1|\lambda',s-1\rangle_{[\bt']}.
\]
\end{enumerate}
(I) follows from the fact that $\ket{\lambda,s}_{[\bt]}$ is a linear combination of vectors of the form $\psi_{n_1}\psi_{n_2}\dots \psi_{n_s}\ket{-r}$ with $N>n_i$ for each $i$.
We prove (II) by induction on $s$.
If $s=0$, (II) is obvious.
For general $s>0$, we have
\begin{align*}
{}_{[\bt]}\bra{\lambda,s}\psi_N
&=
{}_{[\bt]}\bra{\lambda',s-1}e^{-H(t_1)}\psi^\ast_{\mu_1-1}\psi_N\\
&=
-{}_{[\bt]}\bra{\lambda',s-1}e^{-H(t_1)}\psi_N\psi^\ast_{\mu_1-1}\\
&=-{}_{[\bt]}\bra{\lambda',s-1}(\psi_N-t_1\psi_{N-1})e^{-H(t_1)}\psi^\ast_{\mu_1-1}.
\end{align*}
By induction hypothesis, the last expression should be zero, which concludes (II).
(III) is proved as
\begin{equation*}\label{eq:proof_of_(III)}
\begin{aligned}
&{}_{[\bt]}\langle\mu,s|\lambda,s\rangle_{[\bt]}\\
&=
{}_{[\bt']}\bra{\mu',s-1}e^{-H(t_1)}\psi^\ast_{\mu_1-1}\psi_{\lambda_1-1}e^{H(t_1)}\ket{\lambda',s-1}_{[\bt']}\\
&=
{}_{[\bt']}\bra{\mu',s-1}e^{-H(t_1)}(1-\psi_{\lambda_1-1}\psi^\ast_{\mu_1-1})e^{H(t_1)}\ket{\lambda',s-1}_{[\bt']}\\
&=
{}_{[\bt']}\langle\mu',s-1|\lambda',s-1\rangle_{[\bt']}-
{}_{[\bt']}\bra{\mu',s-1}e^{-H(t_1)}\psi_{\lambda_1-1}\psi^\ast_{\mu_1-1}e^{H(t_1)}\ket{\lambda',s-1}_{[\bt']}\\
&=
{}_{[\bt']}\langle\mu',s-1|\lambda',s-1\rangle_{[\bt']}\\
&\hspace{3em}-
\sum_{m=0}^\infty
t_1^m\cdot
{}_{[\bt']}\bra{\mu',s-1}e^{-H(t_1)}\psi_{\lambda_1-1}e^{H(t_1)}\psi^\ast_{\mu_1-1+m}\ket{\lambda',s-1}_{[\bt']}\\
&\stackrel{\mathrm{(I)}}{=}
{}_{[\bt']}\langle\mu',s-1|\lambda',s-1\rangle_{[\bt']}.
\end{aligned}
\end{equation*}
By using (I--III) repeatedly, we conclude the proof.
\end{proof}

Theorem \ref{thm:orthonormalily_of_braket} provides a straightforward method for giving the linear expansion $\ket{\lambda}_{\bx/\by}=\sum_{\mu}C_\lambda^\mu\ket{\mu}_{[\bt]}$.
By using the orthonormality, we obtain the simple formula
\begin{equation}\label{eq:C_mu_lambda}
C_\lambda^\mu={}_{[\bt]}\langle\mu |\lambda\rangle_{\bx/\by},
\end{equation}
which implies the generalized Schur expansion
\begin{equation}\label{eq:expansion_of_multiSchur}
s_\lambda(X)_{\bx/\by}=\sum_\mu {}_{[\bt]}\langle\mu |\lambda\rangle_{\bx/\by}\cdot  s_\mu(X)_{[\bt]}.
\end{equation}

\begin{example}[Schur expansion of $s_\lambda(X)_{\bx/\by}$]
If $t=(0,0,0,\dots)$, then $\ket{\lambda}_{[\bt]}=\ket{\lambda}$.
Assuming $r\geq \max[\ell(\lambda),\ell(\mu)]$, we obtain the determinantial formula
\begin{align*}
C_{\lambda}^\mu=\langle{\mu}|\lambda \rangle_{\bx/\by}
&=
\det\left(
\sum_{l=0}^\infty
h_l(x^{(i)}/y^{(i)})\cdot 
\bra{-r}
\psi^\ast_{\mu_j-j}\psi_{\lambda_i-i-l}
\ket{-r}
\right)_{1\leq i,j\leq r}\\
&=
\det\left(
h_{\lambda_i-\mu_j-i+j}(x^{(i)}/y^{(i)})
\right)_{1\leq i,j\leq r}.
\end{align*}
Since $\lambda_i=0$ for $i>\ell(\lambda)$ and $h_{n}(x^{(i)}/y^{(i)})=0$ for $n<0$, it follows that $\mu\not\subset\lambda\Rightarrow \langle{\mu}|\lambda \rangle_{\bx/\by}=0$.
This fact gives the Schur expansion
\begin{equation}\label{eq:schur_exp_of_s_xy}
s_\lambda(X)_{\bx/\by}=\sum_{\mu\subset \lambda}
\det\left(
h_{\lambda_i-\mu_j-i+j}(x^{(i)}/y^{(i)})
\right)_{1\leq i,j\leq \ell(\lambda)}
s_\mu(X).
\end{equation}
\end{example}

\begin{example}[Schur expansion of refined dual Grothendieck functions]
Let $r\geq \ell(\lambda)$.
By substituting $\bx=[\bt]$ and $\by=\bm{\emptyset}$ to \eqref{eq:schur_exp_of_s_xy}, we have the Schur expansion
\[
g_\lambda(X;t_1,\dots,t_{r-1})=
\sum_{\mu\subset \lambda}
\det\left(
h_{\lambda_i-\mu_j-i+j}(t_1,\dots,t_{i-1})
\right)_{1\leq i,j\leq \ell(\lambda)}
s_\mu(X).
\]
\end{example}

\begin{rem}\label{rem:skew_multi-Schur}
The function 
\[
S_{\lambda/\mu}(\bx/\by):=\det\left(
h_{\lambda_i-\mu_j-i+j}(x^{(i)}/y^{(i)})
\right)_{i,j}
\] 
is called the skew multi-Schur function~\cite[\S 1.4]{lascoux2003symmetric}.
\end{rem}

\begin{example}[Expansion of $s_\lambda(X)_{\bx/\by}$ in refined dual Grothendieck functions]\label{ex:most_general_expansion}
Let $t=(t_1,t_2,\dots)$ and $r\geq \max[\ell(\lambda),\ell(\mu)]$.
Then we have
\begin{align*}
&{}_{[\bt]}\langle{\mu}|\lambda \rangle_{\bx/\by}\\
&=
\det\left(
\sum_{m=0}^\infty
\sum_{l=0}^\infty
h_m(t_j,t_{j+1},\dots,t_r)\cdot
h_l(x^{(i)}/y^{(i)}\cup (t_1,\dots,t_r))\cdot \right.\\
&\hspace{18em}\left.
\bra{-r}
\psi^\ast_{\mu_j-j+m}\psi_{\lambda_i-i-l}
\ket{-r}
\right)_{1\leq i,j\leq r}\\
&=
\det\left(
\sum_{
\substack{
l+m=\lambda_i-\mu_j-i+j\\
0\leq l\leq \lambda_i-i+r
}
}
h_m(t_j,t_{j+1},\dots,t_r)\cdot
h_l(x^{(i)}/(y^{(i)}\cup (t_1,\dots,t_r)))
\right)_{1\leq i,j\leq r}\\
&=
\det\left(
h_{\lambda_i-\mu_j-i+j}
\left(
x^{(i)}/(y^{(i)}\cup (t_1,\dots,t_{j-1}) )
\right)
\right)_{1\leq i,j\leq r},
\end{align*}
which implies $\mu\not\subset \lambda\Rightarrow {}_{[\bt]}\langle{\mu}|\lambda \rangle_{\bx/\by}=0$.
Finally, we have the generalized Schur expansion
\[
\begin{aligned}
&s_{\lambda}(X)_{\bx/\by}=\\
&\sum_{\mu\subset \lambda} \det\left(
h_{\lambda_i-\mu_j-i+j}
\left(
x^{(i)}/(y^{(i)}\cup (t_1,\dots,t_{j-1}) )
\right)
\right)_{1\leq i,j\leq \ell(\lambda)}\cdot g_\mu(X;t_1,\dots,t_{r-1}).
\end{aligned}
\]
\end{example}

\section{Dual symmetric functions}\label{sec:4}

This section aims to study the dual functions of multi-Schur functions.
Unfortunately, since $\Lambda(X)$ is infinite-dimensional, not every linear basis has its dual in $\Lambda(X)$.
Thus we often need to use the extended ring $\widehat{\Lambda}(X)$ (see \cite{iwao2020freefermions}, for example) to obtain desired dual functions.
As the ring $\widehat{\Lambda}(X)$ has a linear topology and is complete, an element of $\widehat{\Lambda}(X)$ can be specified by giving a convergent sequence of symmetric functions.

\subsection{$r$-truncated dual function $s^r_\lambda(X)^\bx$}

We will use the ``dual operator''
\begin{equation*}
H^\ast(x/y):=\sum_{n>0}\frac{p_n(x/y)}{n}a_{-n}
\end{equation*}
to define dual functions.
However, since the sum
\[
e^{-H^\ast(x/y)}\psi_{n}e^{H^\ast(x/y)}=\psi_{n}-e_1(x/y)\psi_{n+1}+e_2(x/y)\psi_{n+2}+\cdots
\]
is finite only if $y=\emptyset$, we should assume $y$ to be empty to have a well-defined linear map from $\mathcal{F}$ to itself.

For $r\geq \ell(\lambda)$, set
\begin{equation}\label{eq:def_of_ket_bx_r}
\begin{split}
\ket{\lambda}_r^{\bx}:=
\left(e^{-H^\ast(x^{(1)})}\psi_{\lambda_1-1}e^{H^\ast(x^{(1)})}\right)
\left(e^{-H^\ast(x^{(2)})}\psi_{\lambda_2-2}e^{H^\ast(x^{(2)})}\right)\\
\cdots
\left(e^{-H^\ast(x^{(r)})}\psi_{\lambda_r-r}e^{H^\ast(x^{(r)})}\right)\ket{-r}.
\end{split}
\end{equation}
Since all the expressions on the right hand side of \eqref{eq:def_of_ket_bx_r} are finite sums, the  $\ket{\lambda}_r^{\bx}$ is contained in the Fock space $\mathcal{F}$.
Let us define the new symmetric function 
\[
s^r_\lambda(X)^{\bx}:=\bra{0}e^{H(X)}\ket{\lambda}_r^{\bx}.
\]
We call $s^r_\lambda(X)^{\bx}$ the \textit{$r$-truncated dual function} (see \cite[Section 3]{iwao2020freefermions}).

\begin{example}[Truncated Grothendieck function]\label{example:Groth}
Let $t=(-\beta,-\beta,\dots)$.
Then $s^r_\lambda(X)^{[\bt]}$ coincides with the $r$-truncated Grothendieck function
$G^r_\lambda(X)$ introduced in \cite[\S 3.1]{iwao2020freefermion}.
Even if $r$ tends to be sufficiently large, the sequence of $r$-truncated Grothendieck functions does not stable: $G^r_\lambda(X)\neq G^{r+1}_\lambda(X)\neq G^{r+2}_\lambda(X)\neq \cdots$.
Instead, the sequence converges to some element in the topological ring $\widehat{\Lambda}(X)$. 
The limit $G_\lambda(X):=\lim\limits_{r\to \infty}G_\lambda^r(X)$ is called the stable Grothendieck polynomial.
\end{example}



\begin{example}[Schur expansion of $s^r_\lambda(X)^{\bx}$]
\label{example:Schur_exp_of_s^x}
For $\bx=(x^{(1)},x^{(2)},\dots)$, let
\[
\overline{x}^{(i)}
=\begin{cases}
x^{(i)}, & 1\leq i\leq r,\\
\emptyset, & i>r
\end{cases}
\]
be a ``$r$-truncated'' tuple.
For any $R\geq \max[\ell(\mu),r]$, we have
\[
\begin{aligned}
\langle \mu|\lambda \rangle_r^{\bx}
&=\det\left(
\sum_{m=0}^\infty
(-1)^me_m(\overline{x}^{(i)})
\bra{-R}\psi^\ast_{\mu_j-j}\psi_{\lambda_i-i+m}\ket{-R}
\right)_{1\leq i,j\leq R}\\
&=
\det\left(
e_{-\lambda_i+\mu_j+i-j}(-\overline{x}^{(i)})
\right)_{1\leq i,j\leq R}.
\end{aligned}
\]
Because $e_n(\emptyset)=\delta_{n,0}$ and $\lambda_i=0$ $(i>r)$, the determinant $\langle \mu|\lambda \rangle_r^{\bx}$ must vanish if $\ell(\mu)>r$.
Hence we have the Schur expansion
\[
s^r_\lambda(X)^{\bx}
=\sum_{
\substack{
\mu\supset \lambda\\
\ell(\mu)\leq r
}
} \det\left(
e_{-\lambda_i+\mu_j+i-j}(-x^{(i)})
\right)_{1\leq i,j\leq r}s_\mu(X).
\]
\end{example}

\begin{example}[Schur expansion of $s^r_\lambda(X)^{[\bt]}$]
Assume $t=(t_1,t_2,\dots)$ and $r\geq \ell(\lambda)$.
We call $G^r_\lambda(X;t_1,\dots,t_{r-1}):=s^r_\lambda(X)^{[\bt]}$ the $r$-truncated refined Grothendieck function.
From Example \ref{example:Schur_exp_of_s^x}, we have the Schur expansion
\[
G^r_\lambda(X;t_1,\dots,t_{r-1})=
\sum_{
\substack{
\mu\supset \lambda\\
\ell(\mu)\leq r
}
}
\det\left(
e_{-\lambda_i+\mu_j+i-j}(-t_1,\dots,-t_{i-1})
\right)_{1\leq i,j\leq r}s_\mu(X).
\]
\end{example}

\subsection{$r$-th stable dual function $s^{[r]}_\lambda(X)^\bx$}

Seeing Example \ref{example:Groth}, one might expect that the sequence 
\begin{equation}\label{eq:sequence}
s^1_\lambda(X)^{\bx},s^2_\lambda(X)^{\bx},\dots
\end{equation}
converges in $\widehat{\Lambda}(X)$, but this is not the case for general $\bx$.
Here we do not go into details of the topology of $\widehat{\Lambda}(X)$ but just recall the following fact:
\begin{lemma}\label{lemma:fact}
Let $f_1(X),f_2(X),\dots\in \Lambda(X)$ be a sequence of symmetric functions.
Then the sequence converges in $\widehat{\Lambda}(X)$ if for any $n>0$ there exists some $M>0$ such that 
\[
i,j>M
\quad\Rightarrow\quad 
f_i(X_1,\dots,X_n,0,0,\dots)=f_j(X_1,\dots,X_n,0,0,\dots).
\]
\end{lemma}
\begin{proof}
See \cite[\S 3.1]{iwao2020freefermions}.
\end{proof}

For $r\geq \ell(\lambda)$, we define the vector
\begin{equation}\label{eq:def_of_ket_bx_st}
\begin{split}
\ket{\lambda}^{\bx}_{[r]}:=
\left(e^{-H^\ast(x^{(1)})}\psi_{\lambda_1-1}e^{H^\ast(x^{(1)})}\right)
\left(e^{-H^\ast(x^{(2)})}\psi_{\lambda_2-2}e^{H^\ast(x^{(2)})}\right)\\
\cdots
\left(e^{-H^\ast(x^{(r)})}\psi_{\lambda_r-r}e^{H^\ast(x^{(r)})}\right)
e^{-H^\ast(x^{(r+1)})}
\ket{-r},
\end{split}
\end{equation}
which ``approximates'' $\ket{\lambda}^\bx_r$.
Note that $\ket{\lambda}^\bx_r$ and $\ket{\lambda}^\bx_{[r]}$ are indeed different since $e^{-H^\ast(x^{(r+1)})}\ket{-r}\neq \ket{-r}$ in general.
Furthermore, $\ket{\lambda}^{\bx}_{[r]}$ is \textit{not always} an element of $\mathcal{F}$ because  $e^{-H^\ast(x^{(r+1)})}\ket{-r}$ may not be a finite sum of vectors \eqref{eq:elementary-vactors}.

Let us introduce the ``formal'' symmetric function
\[
s^{[r]}_\lambda(X)^\bx:=\bra{0}e^{H(X)}\ket{\lambda}_{[r]}^\bx.
\]
In the sequel, we prove that $s^{[r]}_\lambda(X)^\bx$ is an element of $\widehat{\Lambda}(X)$.
\begin{lemma}[{\cite[Lemma 3.3]{iwao2020freefermion}}, {\cite[Lemma 3.7]{iwao2020freefermions}}]\label{lemma:technical_lemma_1}
Let $x=(x_1,\dots,x_p)$ be a tuple of $p$ indeterminates.
Suppose $r>0$.
Then there exists a certain sequence $Y_1,Y_{2},\dots\in \mathcal{A}$ that satisfies the following properties:
\begin{enumerate}
\item 
The vector $e^{-H^\ast(x)}\ket{-r}$ is formally expanded as
\[
e^{-H^\ast(x)}\ket{-r}=
\ket{-r}+\sum_{i=1}^\infty Y_i\ket{-r-i}.
\]
\item For any $m_1,\dots,m_r\in \ZZ$, the vector $\psi_{m_1}\dots \psi_{m_r}Y_i\ket{-r-i}$ is contained in $\mathcal{F}^0_{(r+i+1)}\setminus \mathcal{F}^0_{(r+i)}$.
\item The element $\psi_{-r+p-1}\cdots \psi_{-r+1}\psi_{-r}Y_i$ vanishes.
\end{enumerate}
\end{lemma}
\begin{proof}
The proof is a variant of \cite[Lemma 3.3]{iwao2020freefermion}.
Here we present a sketch of the proof.
Let $Z_r:=\sum_{l=1}^p(-1)^le_l(x_1,\dots,x_p)\psi_{-r+l}$.
Then we have $e^{-H^\ast(x)}\psi_{-n}=(\psi_{-n}+Z_n)e^{-H^\ast(x)}$.
By using this, we can prove that
\[
Y_i:=(\psi_{-r-1}+Z_{r+1})(\psi_{-r-2}+Z_{r+2})\cdots (\psi_{-r-i+1}+Z_{r+i-1})Z_{r+i}
\]
satisfies the desired properties (1--3).
For details, see \cite[\S 3]{iwao2020freefermion}.
\end{proof}
By using Lemma \ref{lemma:technical_lemma_1} (1), we have the infinite series
\begin{equation}\label{eq:series}
\bra{0}e^{H(X)}\ket{\lambda}^{\bx}_{[r]}
=
\bra{0}e^{H(X)}\ket{\lambda}^{\bx}_{r}+
\sum_{i=1}^\infty\bra{0}e^{H(X)}EY_i\ket{-r-i},
\end{equation}
where
\[
E=
\left(e^{-H^\ast(x^{(1)})}\psi_{\lambda_1-1}e^{H^\ast(x^{(1)})}\right)
\cdots
\left(e^{-H^\ast(x^{(r)})}\psi_{\lambda_r-r}e^{H^\ast(x^{(r)})}\right).
\]
Applying Lemma \ref{lemma:technical_lemma_1} (2) to \eqref{eq:series}, we obtain the formal series
\begin{equation}\label{eq:detailed_evaluation}
s_\lambda^{[r]}(X)^{\bx}=s_\lambda^{r}(X)^{\bx}+\sum_{i=1}^\infty o_i(X),\quad \mbox{where}\quad o_i(X_1,\dots,X_{r+i},0,0,\dots)=0.
\end{equation}
From Lemma \ref{lemma:fact}, we find that the right hand side of \eqref{eq:detailed_evaluation} converges in $\widehat{\Lambda}(X)$.
Moreover, we have
\begin{equation}\label{eq:approx_of_truncated_function}
s_\lambda^{[r]}(X_1,\dots,X_r,0,\dots)^\bx=s_\lambda^{r}(X_1,\dots,X_r,0,\dots)^\bx.
\end{equation}
We call $s_\lambda^{[r]}(X)^\bx$ the \textit{$r$-th stable dual function}.

\subsection{Stable dual function $s_\lambda(X)^\bx$}

Next, we consider letting $r$ tend to infinity.

\begin{lemma}[See {\cite[Lemma 3.5]{iwao2020freefermion}}]\label{lemma:technical_lemma_2}
Let $t=(t_1,t_2,\dots,t_p)$ be a tuple.
For $n\in \ZZ$ and $p>0$, define
\[
I_p=\psi_{n}e^{-H^\ast(t_1)}
\psi_{n-1}e^{-H^\ast(t_2)}
\cdots
\psi_{n-p+1}e^{-H^\ast(t_p)}.
\]
Then we have
\begin{enumerate}
\item 
$I_p=
\psi_{n}\psi_{n-1}
\cdots
\psi_{n-p+1}e^{-H^\ast(t_1,\dots,t_p)}$,
\item $I_p\ket{n-p+1}=\ket{n+1}$.
\end{enumerate}
\end{lemma}
\begin{proof}
We prove (1) by induction on $p>0$.
For $p=1$, the equation is obvious.
For general $p>1$, we have
\begin{align*}
&I_{p+1}\\
&=I_{p}\psi_{n-p}e^{-H^\ast(t_{p+1})}\\
&=
\psi_{n}\psi_{n-1}
\cdots
\psi_{n-p-1}e^{-H^\ast(t_1,\dots,t_p)}\psi_{n-p}e^{-H^\ast(t_{p+1})}
\qquad (\mbox{induction hypothesis})
\\
&=
\psi_{n}\psi_{n-1}
\cdots
\psi_{n-p-1}
\left\{
\sum_{l=0}^{p}
(-1)^le_{l}(t)\psi_{n-p+l}
\right\}
e^{-H^\ast(t_1,\dots,t_p)}e^{-H^\ast(t_{p+1})}\\
&=\psi_{n}\psi_{n-1}\dots \psi_{n-p}e^{-H^\ast(t_1,\dots,t_p,t_{p+1})},\qquad\ (\because\psi_m^2=0)
\end{align*}
which implies (1).
(2) is proved by
\begin{align*}
I_p\ket{n-p+1}&
\stackrel{\mathrm{(1)}}{=}\psi_{n}\psi_{n-1}
\cdots
\psi_{n-p+1}e^{-H^\ast(t_1,\dots,t_p)}\ket{n-p+1}\\
&
=\psi_{n}\psi_{n-1}
\cdots
\psi_{n-p+1}\ket{n-p+1}\qquad (\mbox{Lemma \ref{lemma:technical_lemma_1}}.\, (1),(3))\\
&=\ket{n+1}.
\end{align*}
\end{proof}

\begin{lemma}\label{lemma:simple_tuple}
Let $\bx=(x^{(1)},x^{(2)},\dots)$ be a tuple of tuples.
Assume that there exist some $R>0$ and $t_1,t_2,\dots\in k$ such that
\begin{equation}
x^{(R+p)}\setminus x^{(R)}=(t_1,\dots,t_p)\qquad \mbox{for all }\quad p>0
\tag{$\ast$}\label{eq:assumption_ast}
\end{equation}
%
Then we have $\ket{\lambda}^{\bx}_{[r]}=\ket{\lambda}^{\bx}_{[r+1]}=\ket{\lambda}^{\bx}_{[r+2]}=\cdots$ for sufficiently large $r$.
\end{lemma}
\begin{proof}
Let $n=-\max[R,\ell(\lambda)]-1$.
Then if $r>-n$, we have 
\[
\begin{split}
\ket{\lambda}^{\bx}_{[r]}:=
\left(e^{-H^\ast(x^{(1)})}\psi_{\lambda_1-1}e^{H^\ast(x^{(1)})}\right)
\left(e^{-H^\ast(x^{(2)})}\psi_{\lambda_2-2}e^{H^\ast(x^{(2)})}\right)\\
\cdots
\left(e^{-H^\ast(x^{(n-1)})}\psi_{\lambda_{-n-1}+n+1}e^{H^\ast(x^{(n-1)})}\right)
I_p
\ket{-r},
\end{split}
\]
where $p=r+n+1$.
From Lemma \ref{lemma:technical_lemma_2} (2), it follows that $\ket{\lambda}^{\bx}_{[r]}=\ket{\lambda}^{\bx}_{[r+1]}=\ket{\lambda}^{\bx}_{[r+2]}=\cdots$.
\end{proof}

For $\bx$ satisfying the assumption \eqref{eq:assumption_ast} of Lemma \ref{lemma:simple_tuple}, we write  $\ket{\lambda}^{\bx}:=\ket{\lambda}^{\bx}_{[r]}$ for sufficiently large $r$.
We define the \textit{dual stable function} $s_\lambda(X)^\bx$ as
\[
s_\lambda(X)^{\bx}:=\bra{0}e^{H(X)}\ket{\lambda}^{\bx}.
\]
From \eqref{eq:detailed_evaluation}, we have
\[
s_\lambda(X)^\bx=\lim_{r\to \infty} s_\lambda^{r}(X)^\bx=\lim_{r\to \infty} s_\lambda^{[r]}(X)^\bx.
\]

Because any refined tuple $[\bt]$ satisfies the assumption \eqref{eq:assumption_ast}, there always exists the stable dual function $s_\lambda(X)^{[\bt]}$ associated with $[\bt]$.
\begin{rem}\label{rem:stability_of_bra}
The vector ${}_{[\bt]}\bra{\lambda}$ in \eqref{eq:def_of_bra_bt} is just the transpose of $\ket{\lambda}^{[\bt]}_{[r]}$.
This implies the fact that \eqref{eq:def_of_bra_bt} is independent of the choice on sufficiently large $r$.
\end{rem}

\subsection{Duality between $s_\lambda(X)_{[\bt]}$ and $s_\lambda(X)^{[\bt]}$ }

The Hall inner product can be extended continuously to the bilinear map $\widehat{\Lambda}(X)\otimes_k\Lambda(X)\to k$.
Lemma \ref{lemma:fundamental} is also available for the extended bilinear form.
\begin{prop}\label{prop:dual_functions}
We have
$
\langle
s_\lambda(X)^{[\bt]},
s_\mu(X)_{[\bt]}
\rangle
=\delta_{\lambda,\mu}
$.
\end{prop}
\begin{proof}
It is a direct consequence of Lemma \ref{lemma:fundamental} and Theorem \ref{thm:orthonormalily_of_braket}.
\end{proof}

Let ${}^{[\bt]}\bra{\lambda}=(\ket{\lambda}_{[\bt]})^\ast$ be the transposed vector.
Since 
$
{}^{[\bt]}\langle \mu|\lambda \rangle^{[\bt]}=
{}_{[\bt]}\langle \lambda|\mu \rangle_{[\bt]}=\delta_{\lambda,\mu}
$,
we have the Schur expansion
\[
s_\lambda(X)^{\bx}=\sum_\mu {}^{[\bt]}\langle \mu|\lambda \rangle^{\bx}\cdot s_\mu(X)^{[\bt]}.
\]

\begin{example}[Expansion of $s^r_\lambda(X)^{\bx}$ in stable refined Grothendieck polynomials]
Let $\bt=(t_1,t_2,\dots)$, $r\geq \ell(\lambda)$, and $R\geq \max[r,\ell(\mu)]$.
Then the expansion of $\ket{\lambda}_r^{\bx}$ in $\{\ket{\mu}^{[\bt]}\}$ can be calculated as follows:
\[
\begin{aligned}
&{}^{[\bt]}\langle \mu|\lambda \rangle_r^{\bx}\\
&
=\det\left(
\sum_{m=0}^{\infty}\sum_{l=0}^\infty
h_l(t_1,\dots,t_{j-1})
e_m(-\overline{x}^{(i)})
\bra{-R}\psi^\ast_{\mu_j-j-l}\psi_{\lambda_i-i+m}\ket{-R}
\right)_{1\leq i,j\leq R}\\
&=\det\left(
\sum_{
\substack{
m+l=-\lambda_i+\mu_j+i-j\\
0\leq l\leq \mu_j-j+R
}
}
h_l(t_1,\dots,t_{j-1})
e_m(-\overline{x}^{(i)})
\right)_{1\leq i,j\leq R}\\
&=
\det\left(
h_{-\lambda_i+\mu_j+i-j}((t_1,\dots,t_{j-1})/\overline{x}^{(i)})
\right)_{1\leq i,j\leq R}.
\end{aligned}
\]
Hence, we obtain the expansion
\[
\begin{aligned}
&s^r_\lambda(X)^{\bx}=\\
&\sum_{\mu\supset \lambda} \det\left(
h_{-\lambda_i+\mu_j+i-j}((t_1,\dots,t_{j-1})/\overline{x}^{(i)})
\right)_{1\leq i,j\leq \max[r,\ell(\mu)]}
G_\mu(X;t_1,\dots,t_{\ell(\mu)-1}),
\end{aligned}
\]
where $\overline{x}^{(i)}$ is the $r$-truncated tuple of $x^{(i)}$.
\end{example}

\begin{example}[Expansion of $s_\lambda(X)^{\bx}$ in stable refined Grothendieck polynomials]\label{ex:s_to_srG}
Let $\bx$ be a tuple of tuples satisfying the assumption \eqref{eq:assumption_ast} of Lemma \ref{lemma:simple_tuple}.
Let $R$ be the positive integer in \eqref{eq:assumption_ast}.
Then, for $S\geq \max[R,\ell(\mu)]$, we have
\[
\begin{aligned}
&{}^{[\bt]}\langle \mu|\lambda \rangle^{\bx}\\
&
=\det\left(
\sum_{m=0}^{\infty}\sum_{l=0}^\infty
h_l((t_1,\dots,t_{j-1})/x^{(S+1)})
e_m(-x^{(i)}/-x^{(S+1)})\right.\\
&\hspace{16em}
\left.
\bra{-S}\psi^\ast_{\mu_j-j-l}\psi_{\lambda_i-i+m}
\ket{-S}
\right)_{1\leq i,j\leq S}\\
&=\det\left(
\sum_{
\substack{
m+l=-\lambda_i+\mu_j+i-j\\
0\leq l\leq \mu_j-j+S
}
}
h_l((t_1,\dots,t_{j-1})/x^{(S+1)})
e_m(-x^{(i)}/-x^{(S+1)})
\right)_{1\leq i,j\leq S}\\
&=
\det\left(
h_{-\lambda_i+\mu_j+i-j}((t_1,\dots,t_{j-1})/x^{(i)})
\right)_{1\leq i,j\leq S}.
\end{aligned}
\]
Hence, we have
\[
\begin{aligned}
&s_\lambda(X)^{\bx}=\\
&\sum_{\mu\supset \lambda} \det\left(
h_{-\lambda_i+\mu_j+i-j}((t_1,\dots,t_{j-1})/x^{(i)})
\right)_{1\leq i,j\leq \max[R,\ell(\mu)]}
G_\mu(X;t_1,\dots,t_{\ell(\mu)-1}).
\end{aligned}
\]
\end{example}

\begin{example}[stable Grothendieck polynomials]
Let $t=(-\beta,-\beta,\dots)$.
Then $s_\lambda(X)^{[\bt]}$ coincides with the \textit{stable Grothendieck polynomial} $G_\lambda(X)$.
From Example \ref{ex:s_to_srG}, we have the Schur expansion
\begin{align*}
G_\lambda(X)
&=
\sum_{\mu\supset\lambda} \det\left(
h_{-\lambda_i+\mu_j+i-j}(\emptyset/(\overbrace{-\beta,\dots,-\beta}^{i-1}))
\right)_{1\leq i,j\leq \ell(\mu)}
s_\mu(X)\\
&=
\sum_{\mu\supset \lambda} \det\left(
h_{-\lambda_i+\mu_j+i-j}(\emptyset/(\overbrace{-\beta,\dots,-\beta}^{i-1}))
\right)_{1\leq i,j\leq \ell(\mu)}
s_\mu(X)\\
&=
\sum_{\mu\supset \lambda} \det\left(
\binom{i-1}{-\lambda_i+\mu_j+i-j}\beta^{-\lambda_i+\mu_j+i-j}
\right)_{1\leq i,j\leq \ell(\mu)}
s_\mu(X).
\end{align*}
\end{example}

\section{Multi-Schur functions for skew shapes}\label{sec:5}

In this section, we introduce a multi-Schur function corresponding to a skew shape. 
By the fermionic presentation, we can easily prove the \textit{branching formula}, which is a desirable property for skew symmetric functions.

Let $\bp$ be a tuple of tuples satisfying the assumption \eqref{eq:assumption_ast} in Lemma \ref{lemma:simple_tuple}.
\begin{defi}
We define the symmetric function $s_{\lambda/\mu}(X)_{(\bx/\by)/\bp}$ by
\[
s_{\lambda/\mu}(X)_{(\bx/\by)/\bp}={}_{\bp}\bra{\mu}e^{H(X)}\ket{\lambda}_{\bx/\by}.
\]
\end{defi}

Let $X=(X_1,X_2,\dots)$ and $Y=(Y_1,Y_2,\dots)$ be two strings of indeterminates.
For a symmetric function $f(X)=f(X_1,X_2,\dots)$, we write 
\[
f(X,Y):=f(X_1,X_2,\dots,Y_1,Y_2,\dots).
\]

\begin{thm}\label{thm:ramification}
Let $t=(t_1,t_2,\dots)$.
If $\bp=[\bt]$, we have
\begin{equation}\label{eq:refined_ramification}
s_\lambda(X,Y)_{\bx/\by}=\sum_{\mu}s_{\lambda/\mu}(X)_{(\bx/\by)/[\bt]}\cdot s_\mu(Y)_{[\bt]}
\end{equation}
\end{thm}
\begin{proof}
Since any $\ket{v}\in \mathcal{F}^0$ is decomposed as $\ket{v}=\sum_{\mu}\ket{\mu}_{[\bt]}\cdot {}_{[\bt]}\langle \mu|v \rangle$, we have
$s_\lambda(X,Y)_{\bx/\by}
=\bra{0}e^{H(Y)}e^{H(X)}\ket{\lambda}_{\bx/\by}
=\sum_{\mu}\bra{0}e^{H(Y)}\ket{\mu}_{[\bt]}\cdot {}_{[\bt]}\bra{\mu}e^{H(X)}\ket{\lambda}_{\bx/\by}
$, which implies the desired formula.
\end{proof}

If $\bx=\bp=[\bt]$ and $\by=\bm{\emptyset}$, we write 
\begin{equation}\label{eq:our_skew_Schur}
s_{\lambda/\mu}(X)_{[\bt]}:=s_{\lambda/\mu}(X)_{[\bt]/[\bt]}.
\end{equation}
Because ${}_{[\bt]}\bra{0}=\bra{0}$ and
\[
s_{\lambda/\emptyset}(X)_{[\bt]/[\bt]}
={}_{[\bt]}\bra{0}e^{H(X)}\ket{\lambda}_{[\bt]}=\bra{0}e^{H(X)}\ket{\lambda}_{[\bt]}
=s_{\lambda}(X)_{[\bt]},
\]
the notation \eqref{eq:our_skew_Schur} does not cause any confusion.

\begin{cor}[branching formula]\label{cor:branching_formula}
We have 
\[
s_\lambda(X,Y)_{[\bt]}=\sum_{\mu}s_{\lambda/\mu}(X)_{[\bt]}\cdot s_\mu(Y)_{[\bt]}.
\]
\end{cor}
From Corollary \ref{cor:branching_formula}, we would say that ``$s_{\lambda/\mu}(X)_{[\bt]}$ is a skew version of the multi-Schur function''. 
However, we should note that there exists another notion of ``skew multi-Schur functions'' derived from the Jacobi-Trudi formula.
See Remark \ref{rem:skew}.

\begin{rem}\label{rem:skew}
Now assume $X=0$ and $\bp=\bm{\emptyset}$.
In this case, we have
\[
s_{\lambda/\mu}(0,0,\dots)_{(\bx/\by)/\bm{\emptyset}}=
\bra{\mu}e^{H(0,0,\dots)}\ket{\lambda}_{\bx/\by}=
\langle \mu|\lambda\rangle_{\bx/\by}=
s_{\lambda/\mu}(\bx/\by),
\]
where $s_{\lambda/\mu}(\bx/\by)$ is the ``skew multi-Schur function'' in Remark \ref{rem:skew_multi-Schur}.
Obviously, this fact does not fit with our definition of skew multi-Schur function \eqref{eq:our_skew_Schur}.
\end{rem}

\begin{prop}[Determinantial formula for $s_{\lambda/\mu}(X)_{(\bx/\by)/\bp}$]
\label{prop:det_formula_for_skew_multi}
Assume $r\geq \ell(\lambda)$.
Then we have
\[
s_{\lambda/\mu}(X)_{(\bx/\by)/\bp}=
\det\left(
\sum_{m+n=\lambda_i-\mu_j-i+j}
h_m\left(
x^{(i)}/(y^{(i)}\cup p^{(j)})
\right)
h_n(X)
\right)_{1\leq i,j\leq r}.
\]
\end{prop}
\begin{proof}
Assume $r\geq \ell(\lambda)$.
By Wick's theorem \ref{thm:Wick}, we have
\begin{align*}
&s_{\lambda/\mu}(X)_{(\bx/\by)/\bp}\\
&=
\det\left(
\bra{-r}
e^{H(p^{(j)} )}\psi^\ast_{\mu_j-j}e^{-H(p^{(j)} )}
e^{H(X)}
e^{H(x^{(i)}/y^{(i)} )}\psi_{\lambda_i-i}e^{-H(x^{(i)}/y^{(i)} )}
\ket{-r}
\right)\\
&=
\det\left(
\sum_{m=0}^\infty
\sum_{l=0}^\infty
\sum_{n=0}^\infty
h_m(\emptyset/p^{(j)})
h_l(x^{(i)}/y^{(i)})
h_n(X)
\bra{-r}
\psi^\ast_{\mu_j-j+m}
\psi_{\lambda_i-i-l-n}
\ket{-r}
\right)\\
&=
\det\left(
\sum_{
\substack{
m+l+n=\lambda_i-\mu_j-i+j\\
l+n\leq \lambda_i-i+r
}
}
h_m(\emptyset/p^{(j)})
h_l(x^{(i)}/y^{(i)})
h_n(X)
\right)\\
&=
\det\left(
\sum_{m+n=\lambda_i-\mu_j-i+j}
h_m\left(
x^{(i)}/(y^{(i)}\cup p^{(j)})
\right)
h_n(X)
\right).
\end{align*}
\end{proof}

\begin{example}[Jacobi-Trudi formula for $g_{\lambda/\mu}(X;t_1,\dots,t_{r-1})$]
Let $\bx=\bp=[\bt]$ and $\by=\bm{\emptyset}$.
Assume $r\geq \ell(\lambda)$.
Then Proposition \ref{prop:det_formula_for_skew_multi} implies
\begin{align*}
&s_{\lambda/\mu}(X)_{[\bt]}
=\det\left(
\sum_{m=0}^\infty
h_{m}
\left(
(t_1,\dots,t_{i-1})/(t_1,\dots,t_{j-1})
\right)
h_{\lambda_i-\mu_j-i+j-m}(X)
\right)_{1\leq i,j\leq r},
\end{align*}
where
\[
h_{m}
\left(
(t_1,\dots,t_{i-1})/(t_1,\dots,t_{j-1})
\right)
=\begin{cases}
h_m(t_i,t_{i+1},\dots,t_{j-1}), & i<j,\\
\delta_{m,0}, & i=j,\\
e_m(-t_j,-t_{j+1},\dots,-t_{i-1}), & i>j.
\end{cases}
\]
This determinantal formula coincides with the Jacobi-Trudi formula of the refined dual skew Grothendieck function given in \cite[Theorem 4.15]{motegi2020refined}.
Thus we have
\begin{equation}
s_{\lambda/\mu}(X)_{[\bt]}=g_{\lambda/\mu}(X;t_1,\dots,t_{r-1}).
\end{equation}
\end{example}

\begin{example}[Jacobi-Trudi formula for $g_{\lambda/\mu}(X)$]
Put $t=(-\beta,-\beta,\dots)$.
Since $h_m(\overbrace{-\beta,\dots,-\beta}^N)=\binom{m+N-1}{m}(-\beta)^m=\binom{-N}{m}\beta^m$ and $e_m(\overbrace{\beta,\dots,\beta}^N)=\binom{N}{m}\beta^m$, we have the Jacobi-Trudi formula
\begin{align*}
g_{\lambda/\mu}(X)
=\det\left(
\sum_{m=0}^\infty
\binom{j-i}{m}
\beta^m
h_{\lambda_i-\mu_j-i+j-m}(X)
\right).
\end{align*}
\end{example}

\section*{Acknowledgments}

The author is very grateful to Hiroshi Naruse for letting me know the inspiring preprint \cite{motegi2020refined}.
This work is partially supported by JSPS Kakenhi Grant Number 19K03605.

\bibliographystyle{plain}
\bibliography{Groth}

\end{document}